\documentclass[12pt]{amsart}

\usepackage{hyperref}
\usepackage{color}

\usepackage{graphicx}
\usepackage{xypic}

\usepackage{amsmath,amsfonts,amssymb}

\usepackage{pstricks}
\usepackage{pst-plot}

\theoremstyle{plain}
\newtheorem*{theorem*}{Theorem}
\newtheorem*{lemma*} {Lemma}
\newtheorem*{corollary*} {Corollary}
\newtheorem*{proposition*} {Proposition}
\newtheorem*{conjecture*}{Conjecture}

\newtheorem{theorem}{Theorem}[section]
\newtheorem{lemma}[theorem]{Lemma}

\newtheorem{proposition}[theorem]{Proposition}

\newtheorem{question}[theorem]{Question}

\theoremstyle{definition}

\theoremstyle{remark}
\newtheorem*{remark}{Remark}

\newtheorem*{ack}{Acknowledgement}

\theoremstyle{definition}

\def\Z{\Bbb{Z}}
\def\R{\Bbb{R}}
\def\C{\Bbb{C}}

\def\p{\partial}
\def\be{\begin{equation}} \def\ee{\end{equation}}

\def\sm{\setminus}
\def\bp{\begin{pmatrix}}
\def\ep{\end{pmatrix}}
\def\bn{\begin{enumerate}}
\def\en{\end{enumerate}}

\def\ba{\begin{array}}
\def\ea{\end{array}}

\def\fr12{\frac{1}{2}}

\def\ol{\overline}

\def\to{\mathchoice{\longrightarrow}{\rightarrow}{\rightarrow}{\rightarrow}}
\makeatletter
\newcommand{\shortxra}[2][]{\ext@arrow 0359\rightarrowfill@{#1}{#2}}
\def\longrightarrowfill@{\arrowfill@\relbar\relbar\longrightarrow}
\newcommand{\longxra}[2][]{\ext@arrow 0359\longrightarrowfill@{#1}{#2}}

\makeatother

\begin{document}

\title[Topological $\mu$-constant problem]{A note on a topological approach to the $\mu$-constant problem in dimension 2}

\date{\today}

\author{Maciej Borodzik}
\address{Institute of Mathematics, University of Warsaw, Warsaw, Poland}
\email{mcboro@mimuw.edu.pl}

\author{Stefan Friedl}
\address{Mathematisches Institut\\ Universit\"at zu K\"oln\\   Germany}
\email{sfriedl@gmail.com}

\def\subjclassname{\textup{2010} Mathematics Subject Classification}
\expandafter\let\csname subjclassname@1991\endcsname=\subjclassname
\expandafter\let\csname subjclassname@2000\endcsname=\subjclassname

\thanks{The first author is supported by Polish OPUS grant No 2012/05/B/ST1/03195}

\begin{abstract}
We provide an example, which shows that studying homological
and homotopical properties of cobordisms between arbitrary, that is not necessarily negative, graph manifolds 
is not enough to prove the $\mu$-constant conjecture of
L\^e D\~ung Tr\'ang in complex dimension 2.
\end{abstract}

\maketitle


\section{Introduction}
\subsection{Background}
The study of equisingularity is one of the main questions in  singularity theory. The systematic study dates back to Zariski
\cite{Za65a,Za65b,Za68}. One of the milestones is the fact that if the Milnor number $\mu$ is constant
under the deformation of isolated hypersurface singularities in $\C^{n+1}$, and  if $n\neq 2$, then the deformation is topologically trivial. This fact
was proved by L\^e, and L\^e--Ramanujam in the series of papers \cite{Le71,Le72,LR76}. The case $n=1$ is simple and relies on a full
classification of the singularities of plane curves (compare \cite{Za65a}), while the case $n>2$ uses the $h$-cobordism theorem and the fact,
that links of isolated hypersurface singularities in $\C^{n+1}$ for $n>2$ are simply connected.  In particular, the
proof in the case $n>2$ is purely topological. The
problem for $n=2$ has remained open for 40 years.

There were attempts to solving the $\mu$-constant conjecture using a topological approach, i.e. studying the cobordism of links of singularities.
Perron and Shalen in \cite{PS99} proved the $\mu$-constant conjecture under an additional hypothesis on fundamental groups of the links.
They use a detailed study of graph manifolds and a deep understanding of the cobordism between graph manifold. A natural question that arises is:

\smallskip
\emph{Can one prove the $\mu$-constant conjecture in dimension $2$ using only properties of cobordisms of graph manifolds?}

\smallskip
A precise formulation of the above question is given in Question~\ref{qu:estion2}. In this note we show, that if one admits graph manifolds that
are not negative (that is, are not boundaries of negative definite plumbed manifolds), then the answer to Question~\ref{qu:estion2} is negative. As
links of singularities are negative definite,
the counterexample that we give,
does not imply that  topological arguments alone are insufficient to prove the $\mu$-constant conjecture in dimension $2$. However, it indicates
that topological approaches must take into account negative definiteness of graph manifolds.

\subsection{The $\mu$-constant problem}

We begin with the following formulation of the $\mu$-constant problem. We refer to \cite{Te76,GLS06} for background material on
deformations and equisingularity questions.

\begin{question}[The $\mu$-constant problem]\label{qu:estion}
Suppose we are given a family of complex polynomial functions $F_t\colon(\C^{n+1},0)\to(\C,0)$ smoothly depending on
a parameter $t\in D\subset\C$, where $D$ is a unit disk. Assume that for each $t$, the hypersurface $X_t=F_t^{-1}(0)$
has an isolated hypersurface singularity at $0\in\C^{n+1}$. Let $\mu_t$ be the Milnor number of the singularity of $X_t$ at $0$.
If $\mu_t$ is a constant function of $t$, does it imply that the topological type of the singularity of $X_t$ at $0$ does not depend on $t$?
\end{question}

The results \cite{Le71,Le72,LR76} can be resumed as follows.

\begin{theorem}
Question~\ref{qu:estion} has an affirmative answer if $n=1$ or $n\geq 3$.
\end{theorem}

A possible approach to the problem, and actually the one that is sufficient for cases $n\neq 2$ is the following. Let $B_0\subset\C^{n+1}$
be a small closed ball around $0$, such that $M_0:=X_0\cap\partial B_0$ is the link of singularity $(X_0,0)$. Let us pick $t\in\mathbb{C}$ sufficiently small
so that $X_t\cap\partial B_0$ is isotopic to $M_0$. Let us now choose a smaller ball $B_t\subset\C^{n+1}$ such $M_t:=X_t\cap\partial B_t$
is the link of the singularity $(X_t,0)$. Let $W=\ol{X_t\cap(B_0\sm B_t)}$. Then $W$ is a smooth manifold of real dimension $2n$ with
boundary $M_0\sqcup -M_t$.

The cobordism $W$ has various topological properties which we now summarize in the following proposition.

\begin{proposition}\label{prop:W0-W4}
The manifolds $(W,M_0,M_t)$ satisfy the following properties.
\begin{itemize}
\item[(W0)] $\dim_\R W=2n$, $\dim_\R M_0=\dim_\R M_t=2n-1$, furthermore $W,M_0$ and $M_t$ are compact and oriented.
\item[(W1)] If $n>2$, then $\pi_1(M_t)=\pi_1(M_0)=\{e\}$, if $n=2$, then the image of $\pi_1(M_t)$ in $\pi_1(W)$ normally generates $\pi_1(W)$.
\item[(W2)] $W$ can be built from $M_t\times[0,1]$ by adding handles
of indices $0,1,\dots,n$.
\item[(W3)] If $n=2$, then the manifolds $M_0$ and $M_t$ are oriented, irreducible, graph manifolds.
\end{itemize}
Furthermore, if we have the equality of Milnor numbers $\mu_t=\mu_0$, then the following additional fact is satisfied
\begin{itemize}
\item[(W4)] The maps $H_*(M_0;\Z)\to H_*(W;\Z)$ and $H_*(M_t;\Z)\to H_*(W;\Z)$ induced by inclusions are isomorphisms.
\end{itemize}
\end{proposition}

\begin{remark}
Proposition~\ref{prop:W0-W4} is well known to the experts, for a convenience of the reader we sketch the proofs or give references.

\noindent (W0) is obvious. For $n>2$, (W1) is \cite[Theorem 6.4]{Mi68}. (W1) for $n=2$ and (W4) in the general case
can be found in \cite[proof of Theorem 2.1]{LR76}. The main idea is to consider the Milnor fibers $F_t$ and $F_0$
for the singularities $(X_t,0)$ and $(X_0,0)$. By \cite[Theorem 6.5]{Mi68}, $F_t$ has the homotopy type of a wedge of $\mu_t$ spheres $S^n$, and $F_0$
has the homotopy type of a wedge of $\mu_0$ spheres $S^n$.
Using the equivalence of the Milnor fibration over circle and over a disk (see e.g. \cite[Satz 1.5]{Ham71})
we infer that $F_t\cup_{M_t} W$ is homeomorphic $F_0$. Now if $n=2$, then $F_t$ and $F_0$ are simply connected, hence we get (W1) by the van Kampen theorem.

If $\mu_t=\mu_0$, then $F_t$ and $F_0$ have the same homotopy type. Since the homology groups of $F_t$ and $F_0$ are zero in all dimensions but $0$ and $n$,
the standard homological arguments yield (W4).

The property (W2) is proved in \cite{AnF59}. Finally, (W3) follows from \cite{Ne81}.
\end{remark}

As it was written in \cite{LR76},
in case $n>2$, the conditions (W0), (W1) and (W4)
imply that $W$ is an $h$-cobordism and since $\dim_{\R}W\geq 6$ we can appeal to the $h$-cobordism theorem of  Smale (see \cite{Sm62,Mi65})
which shows that $W$ is in fact a product, which in turn
implies that the singularities $(X_t,0)$ and $(X_0,0)$ are topologically equivalent.
In case $n=2$, neither of the manifolds $M_t$ and $M_0$
is simply connected, nor does the Whitney trick work (compare \cite[Section 9.2]{GS99}). However, since
the graph 3-manifolds are somehow rigid, it is still natural, though, to ask the following question.

\begin{question}[Topological $\mu$-constant problem]\label{qu:estion2}
Let $(W,M_t,M_0)$ satisfy conditions $($W0$)$--$($W4$)$ for $n=2$. Does it imply that $M_0$ and $M_t$ are homeomorphic?
\end{question}

As is pointed out in \cite[p.~3]{PS99}, the result of Levine \cite[Theorem 3]{Lv70}
implies, that if $M_0$ and $M_t$ are homeomorphic, then the singular points $(X_t,0)$
and $(X_0,0)$ are topologically equivalent. The key element of this observation is the fact that
$M_0$ and $M_t$ are simple knots by \cite[Lemma 6.4]{Mi68} and the fact, see e.g. \cite[Corollary 1.3]{Sae00},
that $M_0$ and $M_t$ have equivalent Seifert matrices.
In \cite[Proposition~0.5]{PS99} (see also \cite[p.~1180]{AsF11}) the following theorem was proved.

\begin{proposition}
Question~\ref{qu:estion2} has  an affirmative answer if we additionally assume that $\pi_1(M_t)$ surjects onto $\pi_1(W)$.
\end{proposition}

The main goal of this note, and actually the content of next section is the following result.

\begin{theorem}\label{mainthm}
Question~\ref{qu:estion2} has a negative answer.
\end{theorem}


\section{A negative answer to Question~\ref{qu:estion2}}

\subsection{The construction}
Let $K\subset S^3$ be a
non-trivial torus knot $T(p,q)$. We consider $L=K\#-K$, where $-K$ is the mirror image of $K$ with the opposite orientation. Then it is well-known
 (see e.g. \cite[p.~210-213]{GS99}) that $L$ bounds a ribbon disk $D$ in $B^4$.
This implies that there exists an open ball $B'\subset B^4$ with the same center and smaller radius, such that $\p \overline{B'}\cap D$ is an unknot, the distance
function on $D\cap (B^4\sm B')$ is Morse (for this we might need to move slightly the common center of the two balls)
and has only critical points of index $0$ and $1$, and $D\cap (B^4\sm B')$ is an annulus.

Let $\nu D $ be an open tubular neighbourhood of $D$. We define $X=B^4\sm (B'\cup\nu D)$. Let $Y=\p X\cap\nu D$. We define now
\[W=X\cup_Y -X,\]
i.e. we take a double of $X$ along $Y$.

\begin{lemma}
The boundary of $W$ is a disjoint union of $S^1\times S^2$ and the double of $S^3\sm \nu L$.
\end{lemma}

\begin{proof}
We write $S'=\partial \ol{B'}$ and $J:=D\cap S'$. Note that $J$ is the unknot.
It follows immediately from the definitions that $\partial W$ is the disjoint union of
the double of $S^3\sm \nu L$ and the double of $S'\sm \nu J$.
The knot $J\subset S'$ is the unknot, i.e. $S'\sm \nu J$ is a solid torus. The double of $S'\sm \nu J$
is thus canonically homeomorphic to $S^1\times S^2$.
\end{proof}

We define $M_t=S^2\times S^1$ and $M_0$ as the double of $S^3\sm \nu L$.
Since $L$ is non-trivial it is clear that $M_t$ and $M_0$ are non-homeomorphic.
In the next section we will see that the triple $(W,M_t,M_0)$ satisfies conditions (W0) to (W4),
which thus gives us a proof of Theorem \ref{mainthm}.

\subsection{Proof of (W0)--(W4)}

The property (W0) is obvious.
It is a straightforward consequence of Alexander duality and the Mayer--Vietoris sequence that
the maps $H_*(M_0;\Z)\to H_*(W;\Z)$ and $H_*(M_t;\Z)\to H_*(W;\Z)$ induced by inclusions are isomorphisms.
This proves that  (W4) is satisfied.

Let us now show (W2). We use the theory of embedded handle calculus as in \cite[Section 6.2]{GS99}. Namely,
the function `distance from the origin' on  $B^4\sm B'$ has only critical points of index $0$ and $1$, when restricted to $D\cap (B^4\sm B')$.
It follows from \cite[Proposition 6.2.1]{GS99} that $X$ can be built from $\p B'\sm D$ by adding only handles of index $1$ and $2$.
By taking the double we obtain that $W$ is built from $S^2\times S^1$ by adding only handles of index $1$ and $2$ as desired.

We now turn to the proof of (W1). Let $x$ be a generator of $\pi_1(S^1\times S^2)$
which we can represent by a meridian of the unknot $J=D\cap S'$.
We claim that $x$ normally generates $\pi_1(W)$. We denote by $\Gamma$ the smallest normal subgroup
of $\pi_1(W)$ which contains $x$. We thus have to show that in fact $\Gamma=\pi_1(W)$.
First note that the meridian of $J$ is homotopic in $X$, via meridians of the ribbon disk, to a $y$ meridian of the knot $L$.
It is well-known that a meridian normally generates a knot group.
We thus see that $\operatorname{Im}(\pi_1(S^3\sm \nu L)\to \pi_1(W))\subset \Gamma$.
Note that $\pi_1(M_0)$ is generated by the fundamental groups of the two knot exteriors which are glued together.
We now see that  $\operatorname{Im}(\pi_1(M_0)\to \pi_1(W))\subset \Gamma$.
It follows from (W2) that $W$ is obtained from $M_t\times [0,1]$ by adding handles of indices $0,1,2$.
By turning the handle decomposition `upside-down' we see that we can obtain $W$ from $M_0\times [0,1]$
by adding handles of indices $2,3,4$. This implies in particular that $\pi_1(M_0)\to \pi_1(W)$ is surjective.
It now follows that $\pi_1(W)=\Gamma$ as desired.

We finally turn to the proof of (W3).
It is well-known that $S^3\sm \nu K$ is a Seifert fibered space.
Furthermore, we can obtain $S^3\sm \nu L=S^3\sm \nu K\# -{K}$ by
gluing $S^3\sm \nu K$ and $S^3\sm \nu -{K}$ along their boundaries to $S^1\times \Sigma$
where $\Sigma$ is a pair of pants, i.e. $\Sigma$ is obtained by removing three open disks from $S^2$.
It now follows that $S^3\sm \nu L$ is a graph manifold. Finally $M_0$ is obtained by gluing two graph manifolds along their boundary,
which shows that $M_0$ is also a graph manifold.

\subsection{Further properties of $(W,M_t,M_0)$}\label{sec:further}

To conclude, we point out that if the triple $(W,M_t,M_0)$
arises from singularity theory, then the following conditions (W5), (W6) and (W7) are additionally satisfied.
The properties (W6) and (W7) are not of topological nature. We do not know, if all the properties
(W0)--(W7) together, enforce the cobordism $W$ to be a product.

\begin{itemize}
\item[(W5)] The manifolds $M_t$ and $M_0$ are negative, in other words there exists a negative definite plumbing diagram for each of them,
see \cite{Ne81}.
\item[(W6)] The manifold $W$ carries a canonical symplectic form, coming from the K\"ahler structure on it.
The boundary $M_0$ is convex, while $M_t$ is concave with respect to that form. The contact structures induced on $M_t$ and $M_0$ are
Milnor fillable (see \cite{Var80,CNP06}).
\item[(W7)] The manifold $W$ is Stein. Furthermore, $M_t$ is the pseudo-convex part of its boundary and the manifold $M_0$ is pseudo-concave.
\end{itemize}

\begin{ack}
We wish to thank Charles Livingston, Andr\'as N\'emethi and Piotr Przytycki for valuable discussions.
This project was started while we were both visiting Indiana University and we would like to express our gratitude
for hospitality of the Mathematics Department
of Indiana University. We are also grateful to the anonymous referee for carefully reading an earlier version of this paper.
\end{ack}


\begin{thebibliography}{HKL00}

\bibitem[AnF59]{AnF59}
Aldo Andreotti and Theodore Frankel, {\em The Lefschetz theorem on hyperplane sections},
Ann. of Math. (2) 69 (1959), 713--717.

\bibitem[AsF11]{AsF11} Matthias~Aschenbrenner and Stefan~Friedl, \emph{Residual properties of graph manifold groups}, Topology App. \textbf{158} (2011),
no 10, 1179--1191.

\bibitem[CNP06]{CNP06} Cl\'ement Caubel, Andr\'as~N\'emethi and  Patrick Popescu-Pampu,
\emph{Milnor open books and Milnor fillable contact 3-manifolds},
Topology \textbf{45} (2006), no. 3, 673--689.

\bibitem[GS99]{GS99} Robert~Gompf and Andr\'as~Stipsicz,
\textit{4-manifolds and Kirby calculus},
Graduate Studies in Mathematics. \textbf{20}. Providence, RI. 1999.

\bibitem[GLS06]{GLS06} Gert-Martin~Greuel, Christoph Lossen and Eugenii~Shustin,
\textit{Introduction to Singularities and Deformations}, Springer Verlag, Berlin, Heidelberg, New York 2006.

\bibitem[Ham71]{Ham71} Helmut~Hamm, \textit{Lokale topologische Eigenschaften komplexer R\"aume}, Math. Ann. \textbf{191} (1971), 235--252.

\bibitem[Le71]{Le71} L\^e D\~ung Tr\'ang,
\emph{Sur un crit\`ere d'\'equisingularit\'e},  C. R. Acad. Sci. Paris Ser. A--B \textbf{272} (1971), A138--A140.

\bibitem[Le72]{Le72} L\^e D\~ung Tr\'ang,
\emph{Sur les n\oe uds alg\'ebriques},
Compos. Math. \textbf{25} (1972), 281--321.

\bibitem[LR76]{LR76} L\^e D\~ung Tr\'ang and C.~Ramanujam,
\emph{The invariance of Milnor number implies the invariance of topological type},
Amer. J. Math, \textbf{98} (1976), no.1, 67--78.

\bibitem[Lv70]{Lv70} Jerome~Levine, \emph{An algebraic classification of some knots in codimension 2},
Comm. Math. Helv. \textbf{45} (1970), 185--198.

\bibitem[Mi65]{Mi65} John~Milnor, \emph{Lectures on the $h$-cobordism theorem.} Notes by L. Siebenmann and J. Sondow.
Princeton University Press, Princeton, N.J. 1965.

\bibitem[Mi68]{Mi68} John~Milnor, \emph{Singular points of complex hypersurfaces.} Annals of Mathematics Studies, No. 61 Princeton University Press,
Princeton, N.J.; University of Tokyo Press, Tokyo 1968.

\bibitem[Ne81]{Ne81} Walter~Neumann, \emph{A calculus for plumbing applied to the topology of complex surface singularities and degenerating complex curves},
Trans. Am. Math. Soc. \textbf{268} (1981), 299--343.

\bibitem[PS99]{PS99} Bernard~Perron and Peter~Shalen, \emph{Homeomorphic graph manifolds: a contribution to the $\mu$-constant
problem}, Topology Appl. \textbf{99} (1999), no. 1, 1--39.

\bibitem[Sae00]{Sae00} Osamu~Saeki,
\emph{Real Seifert form determines the spectrum for semiquasihomogeneous hypersurface singularities in C3}
J. Math. Soc. Japan \textbf{52} (2000), no. 2, 409--431.

\bibitem[Sm62]{Sm62} Stephen~Smale,
\textit{On the structure of manifolds}
Amer. J. Math. \textbf{84} (1962), 387--399.

\bibitem[Te76]{Te76} Bernard~Teissier, \emph{The hunting of invariants in the geometry of discriminants} in:
Real and complex singularities (Proc. Ninth Nordic Summer School/NAVF Sympos. Math., Oslo, 1976), 565--678.

\bibitem[Var80]{Var80} Alexander~Varchenko, \emph{Contact structures and isolated singularities},
Vestn. Mosk. Univ., Ser. I, 1980, No.2, 18--21.

\bibitem[Za65a]{Za65a} Oscar~Zariski, \emph{Studies in equisingularity. I. Equivalent singularities of plane algebroid curves},
Amer. J. Math. \textbf{87} (1965), 507--536.

\bibitem[Za65b]{Za65b} Oscar~Zariski, \emph{Studies in equisingularity. II. Equisingularity in codimension $1$ (and characteristic zero)},
Amer. J. Math. \textbf{87} (1965) 972--1006.
,
\bibitem[Za68]{Za68} Oscar~Zariski, \emph{Studies in equisingularity. III. Saturation of local rings and equisingularity},
Amer. J. Math. \textbf{90} (1968), 961--1023.

\end{thebibliography}
\end{document}